\documentclass[ECP]{ejpecp}
\usepackage{amsfonts} % replace ECP by EJP if needed.

\usepackage{mathrsfs}

\AUTHORS{%
  Chang-Long~Yao\footnote{Academy of Mathematics and Systems Science, CAS,
Beijing, China.
    \EMAIL{deducemath@126.com}}
  }%AUTHORS

\SHORTTITLE{LLN for critical FPP}

\TITLE{Law of large numbers for critical first-passage percolation on the triangular lattice} %\thanks is optional

\KEYWORDS{critical percolation; first-passage percolation; scaling
limit; conformal loop ensemble; law of large numbers} % Separate items with ;

\AMSSUBJ{60K35} % Separate items with ;
\AMSSUBJSECONDARY{82B43} % Optional, separate items with ;

\SUBMITTED{January 17, 2014} % Edit.
\ACCEPTED{March 14, 2014} % Edit.

\VOLUME{19} \YEAR{2014} \PAPERNUM{18} \DOI{v19-3268}

\ABSTRACT{We study the site version of (independent) first-passage
percolation on the triangular lattice $\mathbb{T}$.  Denote the
passage time of the site $v$ in $\mathbb{T}$ by $t(v)$, and assume
that $P(t(v)=0)=P(t(v)=1)=1/2$.  Denote by $a_{0,n}$ the passage
time from $\textbf{0}$ to $(n,0)$, and by $b_{0,n}$ the passage time
from $\textbf{0}$ to the halfplane $\{(x,y):x\geq n\}$.  We prove
that there exists a constant $0<\mu<\infty$ such that as
$n\rightarrow\infty$, $a_{0,n}/\log n\rightarrow \mu$ in probability
and $b_{0,n}/\log n\rightarrow \mu/2$ almost surely.  This result
confirms a prediction of Kesten and Zhang.  The proof relies on the
existence of the full scaling limit of critical site percolation on
$\mathbb{T}$, established by Camia and Newman.}

\begin{document}

%%%%%%%%%%%%%%%%%%%%%%%%%%%%%%%%%%%%%%%%%%%%%%%%%%%%%%%%%%%%%%%%%%%
%%                                                               %%
%% No need for \maketitle.                                       %%
%%                                                               %%
%%%%%%%%%%%%%%%%%%%%%%%%%%%%%%%%%%%%%%%%%%%%%%%%%%%%%%%%%%%%%%%%%%%

%%%%%%%%%%%%%%%%%%%%%%%%%%%%%%%%%%%%%%%%%%%%%%%%%%%%%%%%%%%%%%%%%%%
%%                                                               %%
%% Please replace what follows by the body of your article       %%
%% (up to the bibliography):                                     %%
%%                                                               %%
%%%%%%%%%%%%%%%%%%%%%%%%%%%%%%%%%%%%%%%%%%%%%%%%%%%%%%%%%%%%%%%%%%%

\section{Introduction}
Standard first-passage percolation (FPP) was introduced by
Hammersley and Welsh \cite{r13} in 1965 as a model of fluid flow
through a random medium.  See \cite{r6} for the basic theory and
Section 2 in \cite{r14} for a summary of recent progress. The usual
setup is that of FPP on lattice $\mathbb{Z}^d$, where i.i.d.
non-negative random variables are assigned to nearest-neighbor edges
in $\mathbb{Z}^d$.  We call this setting the bond version of FPP on
$\mathbb{Z}^d$.  However, unless otherwise stated, we will focus on
the site version of FPP on the triangular lattice $\mathbb{T}$,
which is defined precisely in the following, and the reason will be
explained later. The classic results of FPP are mainly stated for
the bond version of FPP on $\mathbb{Z}^d$, but most of them also
hold for the site version of FPP on $\mathbb{T}$.  Unless otherwise
stated, we just state them directly for the latter in this paper.

Let $\mathbb{T}=(\mathbb{V},\mathbb{E})$ denote the triangular
lattice , where $\mathbb{V}$ is the set of sites, and $\mathbb{E}$
is the set of bonds, connecting adjacent sites.  Let
$\{t(v):v\in\mathbb{V}\}$ be an i.i.d. family of non-negative random
variables with common distribution function $F$. A path is a
sequence of distinct sites connected by nearest neighbor bonds.  A
circuit is a path which starts and ends at the same site and does
not visit the same site twice, except for the starting site.
Sometimes we see the circuit as a simple closed curve consisting of
bonds of $\mathbb{E}$. Given a path $\gamma$, we define its
\textbf{passage time} as
$$T(\gamma):=\sum_{v\in \gamma}t(v).$$
The passage time between two site sets $A,B$ is defined as
$$T(A,B):=\inf \{T(\gamma):\gamma \mbox{ is a path connecting some
site of $A$ with some site of $B$}\},$$ and a time minimizing path
between $A,B$ is called a \textbf{geodesic}.  Denote the origin by
$\textbf{0}$.  Define
\begin{align*}
a_{0,n}&:=\inf\{T(\gamma):\gamma\mbox{ is a path from $\textbf{0}$
to $(n,0)$}\},\\
b_{0,n}&:=\inf\{T(\gamma):\gamma\mbox{ is a path from $\textbf{0}$
to $\{(x,y):x\geq n\}$}\}.
\end{align*}
These are called the point to point and point to line passage times
respectively.  It is well known (Kingman \cite{r9}, Wierman and Reh
\cite{r10}) that if $Et(v)<\infty$, there is a nonrandom constant
$\mu=\mu(F)<\infty$ such that
\begin{equation}\label{e16}
\lim_{n\rightarrow\infty}\frac{a_{0,n}}{n}=\lim_{n\rightarrow\infty}\frac{b_{0,n}}{n}=\mu~~a.s.\mbox{
and in }L^1,
\end{equation}
where $\mu$ is called the \textbf{time constant}.  Kesten \cite{r6}
showed that
\begin{equation}\label{e15}
\mu=0~~\mbox{iff }F(0)\geq p_c(\mathbb{T},\mbox{site})=\frac{1}{2},
\end{equation}
where $p_c(\mathbb{T},\mbox{site})$ is the critical probability for
site percolation on $\mathbb{T}$.  Since there is a transition of
the time constant at $F(0)=p_c$, Kesten and Zhang \cite{r1} call
this ``critical'' FPP.

In this paper, we shall restrict ourselves to a special critical
FPP, that is, we assume that
\begin{equation}\label{e30}
P(t(v)=0)=P(t(v)=1)=\frac{1}{2}.
\end{equation}

Note that we can view this model as the critical site percolation on
$\mathbb{T}$.  Recall that it can be obtained by coloring the faces
of the honeycomb lattice randomly, each cell being open (black) or
closed (white) with probability 1/2 independently of the others. For
this critical FPP, from (\ref{e15}) it is natural to ask whether or
not the sequences in (\ref{e16}) converge to positive limits after
properly normalizing.  We give a historical note related to this
problem here.  Let $\theta$ stands for $a$ or $b$.  In a survey
paper \cite{r4} (see the paragraph right below (3.16P) in
\cite{r4}), Kesten pointed out that the results proved in \cite{r5}
that $E\theta_{0,n}$ lies between two positive multiples of $\log n$
would imply that $\{\theta_{0,n}/\log n\}$ is a tight family,
furthermore, using RSW and FKG, one may show that
$P(\theta_{0,n}\leq \varepsilon \log n)$ is small for small
$\varepsilon$, which implies that any limit distribution of
$\theta_{0,n}/\log n$ has no mass at zero.  Later, Kesten and Zhang
\cite{r1} indicated that the estimates they developed in their paper
can be used to prove a strong law of large numbers (SLLN) for
$b_{0,n}$: $b_{0,n}/Eb_{0,n}\rightarrow 1~~a.s.$ Further, they
expected that $Eb_{0,n}/\log n$ and $Ea_{0,n}/\log n$ converge to
finite, strictly positive limits as $n\rightarrow \infty$. In this
paper, we continue the study from \cite{r1}, the following is our
main theorem:

\begin{theorem}\label{th1}
For the critical FPP satisfying (\ref{e30}) on $\mathbb{T}$, there
exists a constant $0<\mu<\infty$, such that
\begin{align}
&\lim_{n\rightarrow\infty}\frac{a_{0,n}}{\log n}=\mu~~\mbox{in
probability,}\label{th11}\\
&\lim_{n\rightarrow\infty}\frac{b_{0,n}}{\log
n}=\frac{\mu}{2}~~a.s.\label{th12}
\end{align}
Furthermore, the convergence in (\ref{th11}) does not occur almost
surely.
\end{theorem}
\begin{remark}
In fact, using our method one can easily generalize Theorem
\ref{th1} to point to point and point to line passage times along
any given direction, and the limits will coincide with the theorem,
since Camia and Newman's full scaling limit (see Section \ref{s1}
below) is invariant under rotations.  Furthermore, it is expected
that the theorem holds for the classic bond version of FPP on
$\mathbb{Z}^2$. Once the existence of full scaling limit of critical
bond percolation on $\mathbb{Z}^2$ is established, one may derive
the theorem by our strategy.  Also, the limits will be the same as
Theorem \ref{th1} because of the conjectural universality of
critical percolation.
\end{remark}

\begin{remark}
One may consider more general critical FPP on $\mathbb{T}$,  for
example, the distribution function $F$ satisfying the conditions
(1.4)--(1.6) in \cite{r1}.  That is,
\begin{align*}
P(t(v)=0)=\frac{1}{2}, E[t^{\delta}(v)]<\infty\mbox{ for some
}\delta>4, P(0<t(v)<C_0)\mbox{ for some }C_0>0.
\end{align*}
It is expected that Theorem \ref{th1} still holds for the $F$ above
(with $\mu(F)$ as a function of $F$).
\end{remark}

For each $r>0$, let $\mathbb{D}_r$ denote the Euclidean disc of
radius $r$ centered at \textbf{0} and $\partial\mathbb{D}_r$ denote
its boundary.  Let $\mathbb{D}$ denote the unit disk for short.  For
$v\in \mathbb{V}$, let $B(v,r)$ denote the discrete ball of radius
$r$ centered at $v$ in the triangular lattice:
$$B(v,r):=\mathbb{V}\cap\{v+\overline{\mathbb{D}}_r\}.$$
We denote by $\partial B(v,r)$ its boundary, which is the set of
sites in $B(v,r)$ that have at least one neighbor outside $B(v,r)$.
For short, we let $B(r):=B(\textbf{0},r)$.
\begin{remark}
We can express $a_{0,n}$ and $b_{0,n}$ in terms of circuits. For
example, it is easy to see that $a_{0,n}$ and the maximum number of
disjoint closed circuits which separate $\textbf{0}$ and $(n,0)$
differ by at most 2.  Note that with probability 1 there is no
infinite cluster for the critical site percolation on $\mathbb{T}$,
therefore the cluster boundaries form loops.  Now we introduce two
quantities for this model, which are similar as $a_{0,n}$ and
$b_{0,n}$ respectively:
\begin{align*}
a'_{0,n}&:=\mbox{the number of loops which separate $\textbf{0}$ and
$(n,0)$},\\
b'_{0,n}&:=\mbox{the number of loops which separate $\textbf{0}$ and
$\{(x,y):x\geq n\}$}.
\end{align*}
Note that $a'_{0,n}$ is essentially introduced in \cite{r7}.  Using
the strategy in the present paper and the result of \cite{r8}, one
may get the following result, which is analogous to Theorem
\ref{th1} but with explicit limit values:
\begin{align}
\lim_{n\rightarrow\infty}\frac{a'_{0,n}}{\log n}&=\frac{1}{\sqrt{3}\pi}~~\mbox{in probability,}\label{e17}\\
\lim_{n\rightarrow\infty}\frac{b'_{0,n}}{\log
n}&=\frac{1}{2\sqrt{3}\pi}~~a.s.\notag
\end{align}
Furthermore, the convergence in (\ref{e17}) does not occur almost
surely.

The explicit limits above mainly relies on the work of \cite{r8}.
However, it seems very hard to give the explicit value of $\mu$ in
Theorem \ref{th1}. Nevertheless, it need not much work to deduce
that $\mu>1/(2\sqrt{3}\pi)$ from above.  We just give a sketch of
the proof here.  First, let us introduce some notations from
\cite{r8}. Camia and Newman defined the conformal loop ensemble
CLE$_6$ in $\mathbb{D}$ (see Section 3.2 in \cite{r2}), which is
almost surely a countably infinite collection of (oriented)
continuum nonsimple loops and is the scaling limit of the cluster
boundaries of critical site percolation on $\eta\mathbb{T}\cap
\mathbb{D}$ with monochromatic boundary conditions.  We inductively
define $L_k$ to be the outermost loop surrounding $\textbf{0}$ in
$\mathbb{D}$ when the loops $L_1,\ldots,L_{k-1}$ are removed.  Note
that the loops $L_k$ exist for all $k\geq 1$ with probability 1.
Define $A_0=\mathbb{D}$ and let $A_k$ be the component of
$\mathbb{D}\backslash L_k$ that contains $\textbf{0}$.  If $D$ is a
simply connected planar domain and $\textbf{0}\in D$, the conformal
radius of $D$ viewed from $\textbf{0}$ is defined to be
$\mbox{CR}(D):=|g'(\textbf{0})|^{-1}$, where $g$ is any conformal
map from $D$ to $\mathbb{D}$ that sends \textbf{0} to \textbf{0}.
For $k\geq 1$, define
$$B_k:=\log\mbox{CR}(A_{k-1})-\log\mbox{CR}(A_k).$$
From Proposition 1 in \cite{r8}, we know that $B_k,k\geq 1$ are
i.i.d random variables. Furthermore, it is shown (see (2) in
\cite{r8}) that
$$E[B_k]=2\sqrt{3}\pi.$$
A well known consequence of the Schwarz Lemma and the Koebe $1/4$
Theorem (see e.g., Lemma 2.1 and Theorem 3.17 in \cite{r16}, see
also (2.1) in \cite{r17} for a similar application) is that
\begin{equation*}
\frac{\mbox{CR}(A_n)}{4}\leq \mbox{dist}(\textbf{0},L_n)\leq
\mbox{CR}(A_n).
\end{equation*}
Let $N(\varepsilon)$ be the number of loops surrounding \textbf{0}
in $\mathbb{D}\backslash \mathbb{D}_{\varepsilon}$.  From the
definition it is clear that $\log\mbox{CR}(A_n)=-\sum_{k=1}^{n}B_k$.
Combining above issues, one may conclude
\begin{equation}\label{e18}
\lim_{\varepsilon\rightarrow 0}-\frac{N(\varepsilon)}{\log
\varepsilon}=\lim_{\varepsilon\rightarrow
0}-\frac{EN(\varepsilon)}{\log
\varepsilon}=\frac{1}{2\sqrt{3}\pi}~~a.s.
\end{equation}
(\ref{e18}) is an analog of Lemma \ref{l3} below.  One can get
analogs of Lemma \ref{l1} and Lemma \ref{l2} following our method.
Combining these issues, the result would be proved.
\end{remark}
\begin{remark}
Consider the oriented (directed) FPP on $\mathbb{Z}^2$ (see e.g.,
Section 12.8 in \cite{r18} for background). We assign independently
to each bond $e$ i.i.d. passage time $t(e)$. Let $\vec{p}_c$ denote
the critical probability for oriented bond percolation on
$\mathbb{Z}^2$. Assume
$$P(t(e)=0)=\vec{p}_c,~~P(t(e)=1)=1-\vec{p}_c.$$  We denote by
$\vec{T}(\textbf{0},(r,\theta))$ the passage time from \textbf{0} to
$(\lfloor r\sin\theta\rfloor,\lfloor r\cos\theta\rfloor)$ by a
northeast path for $(r,\theta)\in\mathbb{R}^+\times[0,\pi/2]$. Based
on Conjecture 4 in \cite{r11}, we conjecture that there is a
constant $0<\vec{\mu}<\infty$, such that as $r\rightarrow\infty$,
$$\frac{\vec{T}(\textbf{0},(r,\pi/4))}{\log r}\rightarrow\vec{\mu}~~\mbox{in probability.}$$
\end{remark}
\begin{remark}
Camia and Newman's full scaling limit plays a central role in the
proof of Theorem \ref{th1}.  We want to note that the scaling limit
of a critical system may help to show laws of large numbers for many
different variables.  For example, consider the largest winding
angle (the interested reader is referred to \cite{r25} for a more
general discussion and references of winding angles)
$\theta_{max,n}$ of the paths from $\textbf{0}$ to $\partial B(n)$
in Kesten's incipient infinite cluster (IIC) \cite{r24}.
Heuristically, once the existence of an appropriate scaling limit of
IIC on $\mathbb{T}$ is established, one may derive a SLLN for
$\theta_{max,n}$ by our strategy.
\end{remark}
\emph{Idea of the Proof.}  We show a SLLN for
$c_n:=T(\textbf{0},\partial B(n))$, that is, $c_n/\log n\rightarrow
\mu/2~~a.s.$, then Theorem \ref{th1} follows from this easily.
First, using the estimates developed by Kesten and Zhang \cite{r1},
we prove that $c_n/Ec_n\rightarrow 1~~a.s.$  Next, we want to show
$Ec_n/\log n\rightarrow \mu/2,$ which implies the required SLLN
immediately.  For this, we divide the discrete ball $B(n)$ into long
annuli, which have the same shape.  The summation of the passage
times of these annuli approximates $c_n$.  Inspired by Beffara and
Nolin \cite{r3}, we express the passage time of an annulus in terms
of the collection of cluster interfaces (see Fig. \ref{fig2}).  When
the annulus is very large, this quantity can be approximated well by
the passage time defined analogously for the corresponding annulus
with respect to Camia and Newman's full scaling limit \cite{r2} (see
Fig. \ref{fig1}).  For this scaling limit, by the subadditive
ergodic theorem we get a SLLN for the passage times of annuli, which
can be used to approximate the passage times of the large and long
annuli for the discrete model.

Throughout this paper, $C, C_1, C_2, \ldots$ denote positive finite
constants that may change from line to line or page to page
according to the context.
\section{Preliminary results}\label{s1}
We shall use some estimates developed in \cite{r1}.  Let us give
some notations from \cite{r1}.  For $0<m<n$, define the annulus
$$A(m,n):=B(n)\backslash B(m).$$
Let $A(p):=A(2^{p},2^{p+1}),~~p\geq 0.$  Next define for $p\geq 0$
\begin{align*}
m(p)&:=\inf\{t\in\{p,p+1,\ldots\}:A(t)\mbox{ contains an open
circuit surrounding \textbf{0}}\},\\
\mathcal {C}_p&:=\mbox{innermost open circuit surrounding \textbf{0}
in }A(m(p)),\mbox{ and set }\mathcal {C}_{-1}:=\emptyset.
\end{align*}
For $p\geq 0$, define
$$\overline{\mathcal {C}}_p:=\mathcal {C}_p\cup\mbox{interior sites of }\mathcal {C}_p,~~~~\mathcal {F}_p:=\{\mbox{$\sigma$-field generated by} \{t(v):v\in \overline{\mathcal {C}}_p\}\},$$
and let $\mathcal {F}_{-1}$ be the trivial $\sigma$-field.  Let
$\Delta_p=\Delta_{p,q}:=E[T(\textbf{0},\mathcal {C}_{m(q)})|\mathcal
{F}_p]-E[T(\textbf{0},\mathcal {C}_{m(q)})|\mathcal {F}_{p-1}]$.
Then we can write
\begin{equation}\label{e19}
T(\textbf{0},\mathcal {C}_{m(q)})-ET(\textbf{0},\mathcal
{C}_{m(q)})=\sum_{p=0}^{q}\Delta_p.
\end{equation}

Essentially the same as the proof of (2.28), (2.29) in \cite{r1},
one may get the following lemma, we omit the proof here.  Note that
(\ref{e7}) is stronger than (2.29) in \cite{r1}, since the
distribution of the passage time in \cite{r1} is more general than
ours. One needs no new technique here.
\begin{lemma}\label{l4}
There exist constants $C_1,C_2,C_3>0$ such that for $q\geq 1$
\begin{equation}\label{e6}
P(m(p)-p\geq t)\leq \exp(-C_1t),~~t,p\geq 0,
\end{equation}
\begin{equation}\label{e7}
P(|\Delta_p|\geq x)\leq C_2\exp(-C_3x),~~x\geq 0,0\leq p\leq q.
\end{equation}
\end{lemma}

Define $R(m,n):=\{v\in \mathbb{V}:|\arg(v)|<\frac{\pi}{10}\}\cap
A(m,n)$. We say a path $\gamma\subset R(m,n)$ is a crossing path in
$R(m,n)$ if the endpoints of $\gamma$ lie adjacent (Euclidean
distance smaller than $1$) to the rays of argument
$\pm\frac{\pi}{10}$ respectively. By step 3 of the proof of Theorem
5 in \cite{r3}, we obtain the following lemma.

\begin{lemma}\label{l5}
There exist constants $C_1,C_2,K_0>0$, such that for all $k>K_0$ and
$n>m>0$,
\begin{align*}
&P(\mbox{there exist $k\log(n/m)$ disjoint closed crossing
paths in $R(m,n)$})\\
&\quad\leq C_1\exp(-C_2k\log(n/m)).
\end{align*}
\end{lemma}
Observe that $T(\partial B(m),\partial B(n))$ equals the maximal
number of disjoint closed circuits which surround $\textbf{0}$ in
$A(m-1,n)$ (see (2.39) and the Appendix in \cite{r1}), from Lemma
\ref{l5} we immediately get:
\begin{corollary}\label{c1}
There exist constants $C_1,C_2,K_0>0$, such that for all $k>K_0$ and
$n>m>0$,
\begin{equation*}
P(T(\partial B(m),\partial B(n))\geq k\log(n/m))\leq
C_1\exp(-C_2k\log(n/m)).
\end{equation*}
\end{corollary}
Note that (2.48) in \cite{r1} also implies this corollary.  Recall
$c_n=T(\textbf{0},\partial B(n))$.  By RSW, FKG and Corollary
\ref{c1}, one easily obtains the following well-known result, which
is (3.23) in \cite{r5}.
\begin{corollary}\label{c2}
There exist constants $C_1,C_2>0$, such that for all $n\geq 1$,
\begin{equation*}
C_1n\leq Ec_n\leq C_2n.
\end{equation*}
\end{corollary}

\begin{lemma}\label{l1}
\begin{equation*}
\lim_{n\rightarrow\infty}\frac{c_n}{Ec_n}=1~~a.s.
\end{equation*}
\end{lemma}
\begin{proof}
As we have discussed before Theorem \ref{th1}, Kesten and Zhang got
similar result for $b_{0,n}$ in \cite{r1} (see (1.15) in \cite{r1}),
but they did not give the proof.  Now let us use the estimates in
\cite{r1} to prove Lemma \ref{l1}.  First we claim that for each
$\varepsilon\in(0,\frac{1}{4})$, there exist constants
$\delta_1>0,C_6>0$ such that
\begin{equation}\label{e8}
P(|T(\textbf{0},\mathcal {C}_{m(q)})-ET(\textbf{0},\mathcal
{C}_{m(q)})|\geq q^{1-\varepsilon})\leq C_6q^{-(1+\delta_1)}.
\end{equation}
Let us prove this. First we define
$$\widetilde{\Delta}_{p,q}:=\Delta_{p,q}I[m(p)-p\leq C_4\log q],$$
where $C_4>3/C_1$ and $C_1$ is from Lemma \ref{l4}. By (\ref{e19}),
we write
\begin{align*}
P(|T(\textbf{0},\mathcal {C}_{m(q)})-ET(\textbf{0},\mathcal
{C}_{m(q)})|\geq q^{1-\varepsilon})&\leq P(\Delta_{p,q}\neq
\widetilde{\Delta}_{p,q}\mbox{ for some }p\leq
q)\\
&\quad+P(|\sum_{p=0}^{q}\widetilde{\Delta}_{p,q}|\geq
q^{1-\varepsilon}).
\end{align*}
Let us now estimate each term separately.  For the first term,
\begin{align*}
P(\Delta_{p,q}\neq \widetilde{\Delta}_{p,q}\mbox{ for some }p\leq
q)&\leq\sum_{p=0}^{q}P(m(p)-p\geq C_4\log q)\\
&\leq (q+1)\exp(-C_4C_1\log q)\quad\mbox{ by }(\ref{e6}).
\end{align*}
Now we estimate the second term.  Similar as the second half of
Lemma 1 in \cite{r1},  we have:  For any $0\leq p,r\leq q$, if
$|p-r|\geq C_4\log q+2$, then $\widetilde{\Delta}_{p,q}$ and
$\widetilde{\Delta}_{r,q}$ are independent. The proof is omitted
here.  Therefore, by Chebyshev's inequality and (\ref{e7}), there
exists a constant $C_5>0$ such that,
\begin{align}
P\left(|\sum_{p=0}^{q}\widetilde{\Delta}_{p,q}|\geq
q^{1-\varepsilon}\right)&\leq \frac{E[\sum_{p=0}^{q}\widetilde{\Delta}_{p,q}]^4}{q^{4(1-\varepsilon)}}\notag\\
&\leq \frac{24E[\sum_{0\leq p_1\leq\ldots\leq p_4\leq q,|p_1-p_4|\leq 3C_4\log q+6}\prod_{i=1}^4\widetilde{\Delta}_{p_i,q}]}{q^{4(1-\varepsilon)}}\notag\\
&\leq\frac{C_5q^2}{q^{4(1-\varepsilon)}}=
C_5q^{-2+4\varepsilon}\label{e3}
\end{align}
By the bounds of the two terms given above, (\ref{e8}) is proved.
(2.74) in \cite{r1} says that there exist constants
$\delta_2>2,C_7>0$ such that
\begin{equation}\label{e20}
P(|c_{2^q}-T(\textbf{0},\mathcal {C}_{m(q)})|\geq x)\leq
C_7x^{-\delta_2}.
\end{equation}
By (\ref{e8}) and (\ref{e20}), for all large $q$ we have
\begin{align}
P(|c_{2^q}-Ec_{2^q}|\geq 3q^{1-\varepsilon})&\leq P(|c_{2^q}-T(\textbf{0},\mathcal {C}_{m(q)})|\geq q^{1-\varepsilon})\notag\\
&\quad+P(|T(\textbf{0},\mathcal {C}_{m(q)})-ET(\textbf{0},\mathcal {C}_{m(q)})|\geq q^{1-\varepsilon})\notag\\
&\quad+P(|ET(\textbf{0},\mathcal {C}_{m(q)})-Ec_{2^q}|\geq q^{1-\varepsilon})\notag\\
&\leq C_7q^{-(2-2\varepsilon)}+C_6q^{-(1+\delta_1)}\label{e22}.
\end{align}

By RSW, FKG and Corollary \ref{c1}, there exist constants
$C_8,C_9,C_{10}>0$ such that for all $x\geq 1$,
\begin{align}
&P(c_{2^q}-c_{2^{q-1}}\geq x)\notag\\
&\quad\leq P(\mbox{there exists no open circuits surrounding
\textbf{0} in }B(2^q)\backslash
B(2^{q-\lfloor\sqrt{x}\rfloor}))\notag\\
&\quad\quad+P(T(\partial B(2^{q-\lfloor\sqrt{x}\rfloor}),\partial
B(2^q))\geq
x)\notag\\
&\quad\leq\exp(-C_8\sqrt{x})+C_9\exp(-C_{10}x),\label{e4}
\end{align}
where $B(2^{q-\lfloor\sqrt{x}\rfloor}):=B(1)$ for $x>q^2$. Then for
all $q\geq 1$ we get
\begin{equation}\label{e5}
Ec_{2^q}-Ec_{2^{q-1}}\leq C_{11},
\end{equation}
where $C_{11}$ is a universal constant.  Define event
$$\mathcal {A}_q:=\{|c_{2^q}-Ec_{2^q}|\geq 3q^{1-\varepsilon}\}\cup\{c_{2^q}-c_{2^{q-1}}\geq q^{1-\varepsilon}\}.$$
It follows from (\ref{e22}) and (\ref{e4}) that
$\sum_{q=1}^{\infty}P(\mathcal {A}_q)<\infty$.  Then the
Borel-Cantelli lemma implies that, a.s. $\mathcal {A}_q$ happens
only finitely many times as $q\rightarrow\infty$.  For large
$2^{q-1}\leq n\leq 2^q$, from $c_{2^{q-1}}\leq c_n\leq c_{2^{q}}$
and (\ref{e5}) we have
\begin{align*}
&\{|c_n-Ec_n|\geq 9q^{1-\varepsilon}\}\\
&\subset\{|c_n-c_{2^q}|\geq
3q^{1-\varepsilon}\}\cup\{|c_{2^q}-Ec_{2^q}|\geq
3q^{1-\varepsilon}\}\cup\{|Ec_{2^q}-Ec_n|\geq
3q^{1-\varepsilon}\}\\
&\subset\mathcal {A}_q.
\end{align*}
Then we know for each $\varepsilon\in(0,\frac{1}{4})$, as
$n\rightarrow\infty$,
\begin{equation}\label{e23}
|c_n-Ec_n|\leq 9q^{1-\varepsilon}~~a.s.,
\end{equation}
where $2^{q-1}\leq n\leq 2^q$.  By Corollary \ref{c2}, $Ec_n$ lies
between two positive multiples of $q$, then Lemma \ref{l1} follows
from (\ref{e23}).
\end{proof}

As it is well discussed in \cite{r19}, there are several different
ways to describe the scaling limit of critical planar percolation.
In the present paper, we focus on the \textbf{full scaling limit}
constructed by Camia and Newman in \cite{r2}, described in detail
below.

First, we compactify $\mathbb{R}^2$ as usual into
$\dot{\mathbb{R}}^2:=\mathbb{R}^2\cup\{\infty\}\simeq \mathbb{S}^2$.
Let $d_{\mathbb{S}^2}$ be the induced metric on
$\dot{\mathbb{R}}^2$.  We call a continuous map from the circle to
$\mathbb{R}^2$ a loop, and the loops are identified up to
reparametrization by homeomorphisms of the circle with positive
winding.  We equip the space $L$ of loops with the following metric:
$$d_{L}(\ell_1,\ell_2):=\inf_{\phi}\sup_{t\in \mathbb{R}/\mathbb{Z}}d_{\mathbb{S}^2}(\ell_1(t),\ell_2(\phi(t))),$$
where the infimum is taken over all homeomorphisms of the circle
which have positive winding.  Let $\mathcal {L}$ be the space of
countable collections of loops in $L$.  Consider the Hausdorff
topology on $\mathcal {L}$ induced by $d_L$. That is, for
$c_1,c_2\in \mathcal {L}$, let
$$d_{\mathcal {L}}:=\inf\{\varepsilon:\forall \ell_1\in c_1,\exists \ell_2\in c_2\mbox{ such that }d_{L}(\ell_1,\ell_2)\leq \varepsilon\mbox{ and vice versa}\}.$$

For the critical site percolation on $\mathbb{T}$, with probability
1 there is no infinite cluster, therefore the cluster boundaries
form loops.  We orient a loop counterclockwise if it has open sites
on its inner boundary and closed sites on its outer boundary,
otherwise we orient it clockwise.

The following celebrated theorem is shown in \cite{r2}:
\begin{theorem}\label{th2}
As $\eta\rightarrow 0$, the collection of all cluster boundary loops
of critical site percolation on $\eta\mathbb{T}$ converges in law,
under the topology induced by $d_{\mathcal {L}}$, to a probability
distribution on $\mathcal {L}$, which is a continuum nonsimple loop
process.
\end{theorem}

The continuum nonsimple loop process in Theorem \ref{th2} is just
the full scaling limit introduced by Camia and Newman in \cite{r2}.
Since it is also called the \textbf{conformal loop ensemble} CLE$_6$
in \cite{r15} (for the general CLE$_{\kappa},8/3\leq\kappa\leq 8$,
see \cite{r15,r8,r20}), we just call it CLE$_6$ in the present
paper.  Although extracting geometric information is far from being
straightforward from CLE$_6$ (according to \cite{r19}), it was used
to show the uniqueness of the quad-crossing percolation limit in
Subsection 2.3 in \cite{r21} and the existence of the monochromatic
arm exponents in Section 4 in \cite{r3}.  In fact, the key idea of
the proof of (\ref{e25}) is stimulated by the latter.

Several properties of CLE$_6$ are established. For example, if two
loops touch each other and have the same orientation, then almost
surely one loop cannot lie inside the other one.  Conversely, if two
loops of different orientations touch each other, then one has to be
inside the other one.  See \cite{r2} for more details.
\begin{figure}
\begin{center}
\includegraphics[height=0.4\textwidth]{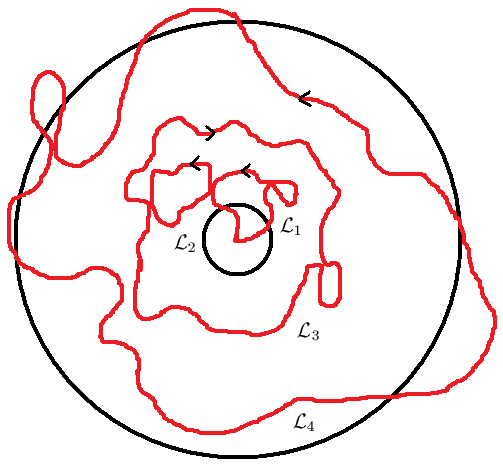}
\caption{A chain $\mathcal {C}=\ell_1\ell_2\ell_3\ell_4$ connecting
two circles. It is easy to see that $T(\mathcal {C})=2$}\label{fig1}
\end{center}
\end{figure}
For CLE$_6$, we want to define the passage time between two circles.
First, we call a sequence of loops $\mathcal {C}=\ell_1\ldots\ell_l$
a \textbf{chain} which connects $\partial \mathbb{D}_m$ and
$\partial \mathbb{D}_n$, if $\mathcal {C}$ satisfies the following
conditions:
\begin{itemize}
\item  $\ell_1\cap \mathbb{D}_m\neq \emptyset, \ell_l\cap \{\mathbb{R}^2\backslash \mathbb{D}_n\}\neq \emptyset$, $\ell_i\subset \mathbb{D}_n\backslash\mathbb{D}_m, 1<i<l$.
\item  For $1\leq i\leq l-1$, if $\ell_i$ is counterclockwise,
then $\ell_{i+1}$ touches $\ell_i$.
\item  For $1\leq i\leq l-1$, if $\ell_i$ is clockwise, then $\ell_{i+1}$ is the minimal counterclockwise loop surrounding $\ell_i$.
\end{itemize}
See Fig. \ref{fig1}. Define the \textbf{passage time} of chain
$\mathcal {C}$ as
$$T(\mathcal {C}):=\mbox{the number of occurrences that $\ell_{i+1}$ touches counterclockwise loop $\ell_{i}$}.$$
The passage time between $\partial \mathbb{D}_m$ and $\partial
\mathbb{D}_n$ is defined as
$$T(\partial \mathbb{D}_m,\partial \mathbb{D}_n):=\inf\{T(\mathcal {C}):\mathcal {C}\mbox{ is chain connecting $\partial \mathbb{D}_m$ and $\partial \mathbb{D}_n$}\}.$$
From (\ref{e25}) below, $T(\partial \mathbb{D}_m,\partial
\mathbb{D}_n)$ is finite with probability $1$.  For the passage time
of this continuum model, we have a strong law of large numbers:
\begin{lemma}\label{l3}
There exists a constant $0<\mu_0<\infty$ such that
\begin{equation*}
\lim_{j\rightarrow\infty}\frac{T(\partial \mathbb{D}_1,\partial
\mathbb{D}_{2^j})}{j}=\lim_{j\rightarrow\infty}\frac{ET(\partial
\mathbb{D}_1,\partial \mathbb{D}_{2^j})}{j}=\mu_0~~a.s.
\end{equation*}
\end{lemma}

\begin{proof}
For short, let $X_{i,j}:=-T(\partial \mathbb{D}_{2^i},\partial
\mathbb{D}_{2^j})+1,0\leq i<j$.  Now we verify that $X_{i,j},0\leq
i<j$ satisfy the conditions of the subadditive ergodic theorem (see
\cite{r12}):
\begin{itemize}
\item $X_{0,j}\leq X_{0,i}+X_{i,j}.$

If $T(\partial \mathbb{D}_1,\partial \mathbb{D}_{2^j})>0$, then for
any chain $\mathcal {C}=\ell_1\ldots\ell_{l}$ connecting $\partial
\mathbb{D}_1$ and $\partial \mathbb{D}_{2^j}$, clearly we have
$l\geq 2$ by the definition of chain. Then it is easy to see that we
can find some $2\leq k\leq l$ such that $\mathcal
{C}_1=\ell_1\ldots\ell_k$ is a chain connecting $\partial
\mathbb{D}_1$ and $\partial \mathbb{D}_{2^i}$, and $\mathcal
{C}_2=\ell_{k-1}\ldots\ell_{l}$ is a chain connecting $\partial
\mathbb{D}_{2^i}$ and $\partial \mathbb{D}_{2^j}$. Therefore,
$$T(\partial \mathbb{D}_1,\partial \mathbb{D}_{2^i})+T(\partial \mathbb{D}_{2^i},\partial \mathbb{D}_{2^j})\leq T(\mathcal {C}_1)+T(\mathcal {C}_2)\leq T(\mathcal {C})+1,$$
which implies the above inequality.  If $T(\partial
\mathbb{D}_1,\partial \mathbb{D}_{2^j})=0$, the inequality holds
obviously.
\item $\{X_{jk,(j+1)k},j\geq 1\}$ is stationary ergodic sequence for each $k$.

Define the scaling transformation
$\tau_k:\mathbb{R}^2\rightarrow\mathbb{R}^2,~~\textbf{x}\mapsto
\textbf{x}/2^k$.  Then for each configuration $\omega$ of CLE$_6$,
$X_{jk,(j+1)k}(\omega)=X_{k,2k}(\tau_k^{j-1}\omega)$. Since CLE$_6$
is invariant under scalings, $\tau_k$ is measure preserving and
$\{X_{jk,(j+1)k},j\geq 1\}$ is stationary. Now we show that $\tau_k$
is also mixing, which implies $\{X_{jk,(j+1)k},j\geq 1\}$ is
ergodic.  When $A,B$ are events which depend only on the realization
of the CLE$_6$ inside an annulus, then
$\lim_{j\rightarrow\infty}P(A\cap\tau_k^{-j}B)=P(A)P(B)$ follows
immediately.  For arbitrary events $A$ and $B$, one approximates $A$
and $B$ by events which depend only on the realization of CLE$_6$
inside the annulus $\mathbb{D}_{1/\varepsilon}\backslash
\mathbb{D}_{\varepsilon}$, and let $\varepsilon\rightarrow 0$.  Then
the result follows easily.
\item The distribution of $\{X_{i,i+k},k\geq 1\}$ does not depend on $i$.

CLE$_6$ is invariant under scalings, which implies this immediately.
\item $EX^+_{0,1}<\infty$, where $X^+_{0,1}:=\max\{X_{0,1},0\}$.  For each $j$, $EX_{0,j}\geq -Cj$, where $C<\infty$.

It is obvious that $EX^+_{0,1}\leq 1$.  From (\ref{e25}) below we
know that the discrete passage time $T(\partial B(n),\partial
B(2^jn))\rightarrow_{d} T(\partial \mathbb{D}_1,\partial
\mathbb{D}_{2^j})$ as $n\rightarrow\infty$. Then by Corollary
\ref{c1}, there exist constants $C_1,C_2,C_3>0$, such that for all
$j\geq 1$,
$$P(T(\partial \mathbb{D}_1,\partial \mathbb{D}_{2^j})\geq C_1j)\leq C_2\exp(-C_3j),$$
which ends the proof immediately.
\end{itemize}
Then by the subadditive ergodic theorem, there exists a constant
$0<\mu_0<\infty$ such that
$$\lim_{j\rightarrow\infty}\frac{X_{0,j}}{j}=\lim_{j\rightarrow\infty}\frac{EX_{0,j}}{j}=-\mu_0~~a.s.,$$
which ends the proof.
\end{proof}

\begin{lemma}\label{l2}
\begin{equation*}
\lim_{n\rightarrow\infty}\frac{Ec_n}{\log n}=\frac{\mu}{2}.
\end{equation*}
\end{lemma}

\begin{proof}
For short, define
\begin{align*}
&A(k,i):=B(2^{k(i+1)})\backslash B(2^{ki}),\quad k\geq 1,i\geq 0.\\
&T_{k,i}:=T(\partial B(2^{ki}+1),\partial B(2^{k(i+1)})),\quad k\geq
1,i\geq 0.
\end{align*}
Recall the definition of $T(\partial \mathbb{D}_m,
\partial \mathbb{D}_n)$ before Lemma \ref{l3}.  For the passage time
defined respectively for the discrete FPP and CLE$_6$, we claim that
for any fixed $k\geq 1$, as $i\rightarrow\infty$,
\begin{equation}\label{e25}
T_{k,i}\rightarrow_d T(\partial \mathbb{D}_1,\partial
\mathbb{D}_{2^k}).
\end{equation}
The proof of this claim is similar as the arguments in Section 4 in
\cite{r3}, but it's more complicated.  First we show that for each
$0<\varepsilon<1$, there is a $\delta>0$, such that
\begin{align}
\lim_{i\rightarrow\infty}P(&\mbox{for any geodesic $\gamma$ connecting $\partial B(2^{ki}+1)$ and $\partial B(2^{k(i+1)})$ in $A(k,i)$,}\notag\\
&\mbox{for any closed site $x\in \gamma$, dist$(x,\partial B(2^{ki}+1)\cup\partial B(2^{k(i+1)}))\geq \delta 2^{ki}$,}\notag\\
&\mbox{for any two closed sites $x,y\in\gamma$, dist$(x,y)\geq
\delta 2^{ki}$})\geq 1-\varepsilon\label{e12}.
\end{align}
Observe that
$$T_{k,i}=\{\mbox{maximal number of disjoint closed circuits which surround \textbf{0} in $A(k,i)$}\}.$$
Therefore, if $\gamma$ is a geodesic connecting $\partial
B(2^{ki}+1)$ and $\partial B(2^{k(i+1)})$ in $A(k,i)$, then there
exist $T(\gamma)$ disjoint circuits which surround \textbf{0} in
$A(k,i)$ and pass through the $T(\gamma)$ closed sites in $\gamma$
respectively.  Using the fact that the polychromatic half-plane
$3$-arm exponent is $2$ (in fact, one needs a more general version,
see Lemma 6.8 in \cite{r22}) and the polychromatic plane $6$-arm
exponent is larger that $2$ (see e.g. \cite{r23}), one can get
(\ref{e12}) by standard arguments.

Define
\begin{align*}
T'_{k,i}:=\inf\{&|\mathcal {S}|-1:\mathcal {S}\mbox{ is a sequence of open clusters, such that the first cluster}\\
&\mbox{intersects with $\partial B(2^{ki}+1)$, the last cluster intersects with $\partial B(2^{k(i+1)})$,}\\
&\mbox{and two consecutive clusters are separated by only one closed
site}\},
\end{align*}
if there exists no such $\mathcal {S}$, we let $T'_{k,i}=0$. From
(\ref{e12}), it is easy to see that
\begin{equation}\label{e13}
\lim_{i\rightarrow\infty}P(T_{k,i}=T'_{k,i})=1.
\end{equation}
\begin{figure}
\begin{center}
\includegraphics[height=0.5\textwidth]{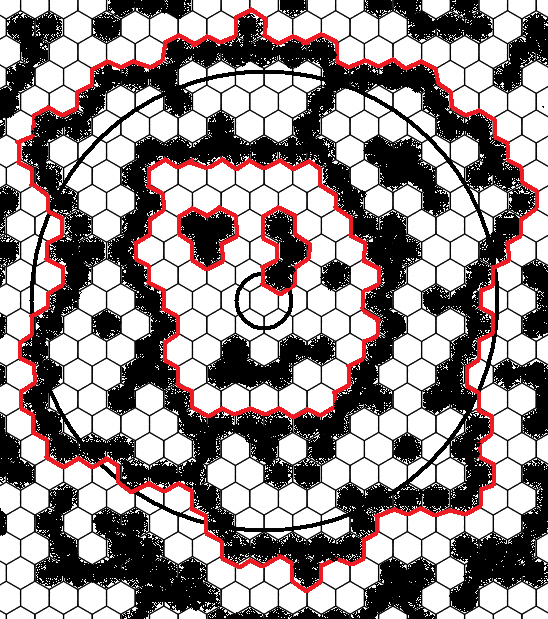}
\caption{A discrete chain $\mathcal {C}=\ell_1\ell_2\ell_3\ell_4$
connecting $\partial B(1)$ and $\partial B(8)$. It is clear that
$T(\mathcal {C})=2$}\label{fig2}
\end{center}
\end{figure}
Now let us introduce the definition of chain for the critical site
percolation on $\mathbb{T}$, which is analogous to its continuum
version for CLE$_6$. We call a sequence of (discrete) loops
$\mathcal {C}=\ell_1\ldots\ell_{l}$ a (discrete) \textbf{chain}
which connects $\partial B(m)$ and $\partial B(n)$, if $\mathcal
{C}$ satisfies the following conditions:
\begin{itemize}
\item  $\ell_1\cap B(m)\neq \emptyset, \ell_l\cap \{\mathbb{R}^2\backslash B(n)\}\neq \emptyset$, $\ell_i\subset B(n)\backslash B(m), 1<i<l$.
\item  For $1\leq i\leq l-1$, if $\ell_i$ is counterclockwise,
then $\ell_{i+1}$ and $\ell_i$ are separated by only one site.
\item  For $1\leq i\leq l-1$, if $\ell_i$ is clockwise, then $\ell_{i+1}$ is the minimal counterclockwise loop surrounding $\ell_i$.
\end{itemize}
See Fig. \ref{fig2}.  For a discrete chain $\mathcal {C}$, let
\begin{align*}
T(\mathcal {C}):=&\mbox{the number of occurrences that $\ell_{i+1}$
and counterclockwise loop $\ell_{i}$ are}\\
&\mbox{separated by only one site}.
\end{align*}
Define
$$T''_{k,i}:=\inf\{T(\mathcal {C}):\mathcal {C}\mbox{ is a chain
connecting $\partial B(2^{ki}+1)$ and $\partial B(2^{k(i+1)})$}\},$$
if there exists no chain connecting $\partial B(2^{ki}+1)$ and
$\partial B(2^{k(i+1)})$, let $T''_{k,i}=0$.  It is easy to get that
\begin{equation}\label{e14}
\lim_{i\rightarrow\infty}P(T'_{k,i}=T''_{k,i})=1.
\end{equation}
By (\ref{e12}),(\ref{e13}) and (\ref{e14}), the value of $T''_{k,i}$
is determined by macroscopic loops with high probability as
$i\rightarrow\infty$. It has been argued in \cite{r3}, two loops
touch in the scaling limit is exactly the asymptotic probability
that they are separated by exactly one site on discrete lattice.
Therefore, using Theorem \ref{th2}, comparing the definitions of
$T''_{k,i}$ and $T(\partial \mathbb{D}_1,\partial
\mathbb{D}_{2^k})$, we have
\begin{equation}\label{e24}
T''_{k,i}\rightarrow_d T(\partial \mathbb{D}_1,\partial
\mathbb{D}_{2^k}).
\end{equation}
Combining (\ref{e13}), (\ref{e14}) and (\ref{e24}), claim
(\ref{e25}) follows.  By Corollary \ref{c1}, there exists a constant
$C(k)>0$, such that for all $i\geq 0$, $ET_{k,i}\leq C(k)$.   This
and (\ref{e25}) immediately give
\begin{equation}\label{e26}
ET_{k,i}\rightarrow ET(\partial \mathbb{D}_1,\partial
\mathbb{D}_{2^k}).
\end{equation}
By the convergence of the Ces\`{a}ro mean and (\ref{e26}), we have
\begin{equation}\label{e1}
\lim_{j\rightarrow\infty}\frac{\sum_{i=0}^j ET_{k,i}}{j}=ET(\partial
\mathbb{D}_1,\partial \mathbb{D}_{2^k}).
\end{equation}
Now let us show that for each $0<\varepsilon<1$, there exists
$k_0(\varepsilon)>0$, such that for each $k\geq k_0$, for $n$
sufficiently large (depending on k),
\begin{equation}\label{e2}
\frac{\sum_{i=0}^{\lfloor\log_{2^k} n\rfloor} ET_{k,i}}{Ec_n}\geq
1-\varepsilon.
\end{equation}
For $i\geq 0,k\geq1$, denote by $N_{k,i}$ the maximum number of
disjoint closed circuits which surround \textbf{0} and intersect
with $\partial B(2^{ki})$.  It is easy to see that
\begin{equation}\label{e27}
\sum_{i=0}^{\lfloor\log_{2^k} n\rfloor-1} T_{k,i}\leq c_n\leq
2+\sum_{i=0}^{\lfloor\log_{2^k} n\rfloor+1} (T_{k,i}+N_{k,i}).
\end{equation}
By RSW, FKG and BK inequality, using standard argument, we know
there exists a constant $C_1>0$, such that for all $i\geq 0,k\geq
1$, $P(N_{k,i}\geq x)\leq \exp(-C_1x)$.  Hence there is a universal
constant $C_2>0$ (independent of $i,k$), such that $EN_{k,i}<C_2$.
Then by (\ref{e27}), there exists a constant $C_3>0$, such that for
all large $n$,
\begin{equation}\label{e28}
\sum_{i=0}^{\lfloor\log_{2^k} n\rfloor} ET_{k,i}-C(k)\leq Ec_n\leq
\sum_{i=0}^{\lfloor\log_{2^k} n\rfloor} ET_{k,i}+C_3\log_{2^k} n,
\end{equation}
where $C(k)$ is introduced before (\ref{e26}).  Combining Corollary
\ref{c2} and (\ref{e28}) gives (\ref{e2}). Then by (\ref{e1}),
(\ref{e2}) and Lemma \ref{l3},
$$\lim_{n\rightarrow\infty}\frac{Ec_n}{\log n}=\lim_{k\rightarrow\infty}\lim_{n\rightarrow\infty}\frac{\sum_{i=0}^{\lfloor\log_{2^k} n\rfloor} ET_{k,i}}{\log n}=\lim_{k\rightarrow\infty}\frac{ET(\partial \mathbb{D}_1,\partial \mathbb{D}_{2^k})}{k\log 2}=\frac{\mu_0}{\log 2}:=\frac{\mu}{2},$$
where $\mu_0$ is from Lemma \ref{l3}.
\end{proof}

\section{Proof of theorem}
\begin{proof}[Proof of Theorem \ref{th1}]
From Lemma \ref{l1} and Lemma \ref{l2}, we have
\begin{equation}\label{e10}
\lim_{n\rightarrow\infty}\frac{c_n}{\log n}=\frac{\mu}{2}~~a.s.
\end{equation}
Now let us use (\ref{e10}) to show $(\ref{th12})$.  First, it is
apparent that $b_{0,n}\geq c_n$.  Thus if one can show that for any
given $\varepsilon>0$, as $n\rightarrow\infty$,
\begin{equation}\label{e11}
b_{0,n}-c_n\leq \varepsilon \log n~~a.s.,
\end{equation}
then (\ref{e10}) and (\ref{e11}) imply (\ref{th12}). We proceed to
prove (\ref{e11}). Recall the definition of $m(p)$ before Lemma
\ref{l4}.  By (\ref{e6}) and Lemma \ref{l5}, we can choose a small
constant $\delta(\varepsilon)
>0$, such that for all large $q$,
\begin{align*}
&P(b_{0,2^q}-c_{2^q}\geq (\varepsilon\log 2^q)/3)\\
&\leq P( m(\lfloor(1-\delta)q\rfloor)\geq q)\\
&\quad+P(\mbox{there exist $\lfloor(\varepsilon\log 2^q)/3\rfloor$
disjoint closed crossing
paths in }R(2^{\lfloor(1-\delta)q\rfloor},2^{q+1}))\\
&\leq \exp(-C_1q)+C_2\exp(-C_3q).
\end{align*}
By (\ref{e4}),
\begin{equation*}
P(c_{2^q}-c_{2^{q-1}}\geq (\varepsilon\log 2^q)/3)\leq
\exp(-C_4\sqrt{q})+C_5\exp(-C_6q).
\end{equation*}
Define events: $\mathcal {A}_q:=\{b_{0,2^q}-c_{2^q}\geq
(\varepsilon\log 2^q)/3\}\cup\{c_{2^q}-c_{2^{q-1}}\geq
(\varepsilon\log 2^q)/3\}$.  Then the inequalities above and
Borel-Cantelli lemma implies that
\begin{equation}\label{e29}
\mbox{a.s. }\mathcal {A}_q\mbox{ happens only finitely many times as
}q\rightarrow\infty.
\end{equation}
Both $\{b_{0,n},n\geq 1\}$ and $\{c_n,n\geq 1\}$ are increasing
sequences, which implies that for $2^{q-1}<n\leq 2^q$
$$b_{0,n}-c_n\leq b_{0,2^q}-c_{2^q}+c_{2^q}-c_{2^{q-1}}.$$
Combining this and (\ref{e29}) gives (\ref{e11}).

(2.84) in \cite{r1} essentially tells us that as
$n\rightarrow\infty$,
$$\frac{1}{\sqrt{\log n}}[T(\textbf{0},(n,0))-T(\textbf{0},\partial B(n/2))-T((n,0),\partial B((n,0),n/2))]\rightarrow 0~~\mbox{in probability}.$$
Combining this and Lemma \ref{l2} gives (\ref{th11}).

Now let us explain why the convergence in (\ref{th11}) does not
occur almost surely.  To show this, with probability $1$, we find a
subsequence that converges to $\frac{\mu}{2}$ in the following. For
$i\geq 1$, define event
\begin{align*}
\mathcal {B}_i:=\{&\mbox{there exists an open circuit surrounding $\textbf{0}$ in $A(2^i,2^{i+1})$, with}\\
&\mbox{an open path connecting it to $\partial B(2^{i+1})$}\}.
\end{align*}
By RSW and FKG, there is a universal constant $C>0$, such that
$P(\mathcal {B}_i)>C$. Then with probability $1$ we can find an
infinite sequence $\{i_j,j\geq 1\}$ such that $A_{i_j}$ happens.
Conditioned on $A_{i_j}$, there exists a $2^{i_j}<n(i_j)\leq
2^{i_j+1}$, such that $a_{0,n(i_j)}=c_{2^{i_j+1}}$.  Then by
(\ref{e10}) we have
$$\lim_{j\rightarrow\infty}\frac{a_{0,n(i_j)}}{\log n(i_j)}=\frac{\mu}{2}~~a.s.$$
This completes the proof.
\end{proof}

%%%%%%%%%%%%%%%%%%%%%%%%%%%%%%%%%%%%%%%%%%%%%%%%%%%%%%%%%%%%%%%%%%%
%%                                                               %%
%% Use the two commands below for producing your bibliography    %%
%% with bibtex, then comment again the commands and include the  %%
%% content of the .bbl file in this file below the commands.     %%
%%                                                               %%
%%%%%%%%%%%%%%%%%%%%%%%%%%%%%%%%%%%%%%%%%%%%%%%%%%%%%%%%%%%%%%%%%%%

%\bibliographystyle{amsplain}
%\bibliography{yourbibfilename}

% add below the content of your .bbl file produced by bibtex.

\pdfbookmark[1]{References}{}

%%%%%%%%%%%%%%%%%%%%%%%%%%%%%%%%%%%%%%%%%%%%%%%%%%%%%%%%%%%%%%%%%%%
%%                                                               %%
%% You may add acknowledgments (optional).                       %%
%%                                                               %%
%%%%%%%%%%%%%%%%%%%%%%%%%%%%%%%%%%%%%%%%%%%%%%%%%%%%%%%%%%%%%%%%%%%

%%%%%%%%%%%%%%%%%%%%%%%%%%%%%%%%%%%%%%%%%%%%%%%%%%%%%%%%%%%%%%%%%%%
%%                                                               %%
%% You have reached the end of your document.                    %%
%%                                                               %%
%%%%%%%%%%%%%%%%%%%%%%%%%%%%%%%%%%%%%%%%%%%%%%%%%%%%%%%%%%%%%%%%%%%

\end{document}